\documentclass[12pt]{article}
\usepackage{graphicx,color}
\usepackage{amsmath,amsfonts,amssymb,amsthm}
\usepackage{rsfso}
\usepackage{comment}
\usepackage{enumerate}

\usepackage[numbers]{natbib}
\newcommand{\ds}{\displaystyle }

\newcommand{\vc}[1]{{\boldsymbol #1}} 
\newcommand{\vcn}[1]{{\bf #1}}

\newcommand{\sr}[1]{{\mathcal #1}}
\newcommand{\dd}[1]{\mathbb{#1}}

\newcommand{\br}[1]{\langle #1 \rangle}
\newcommand{\ol}{\overline}

\newcommand{\all}{{\rm all}}

\newcommand{\eq}[1]{(\ref{eq:#1})}
\newcommand{\lem}[1]{Lemma~\ref{lem:#1}}

\newcommand{\thr}[1]{Theorem~\ref{thr:#1}}

\newcommand{\rem}[1]{Remark~\ref{rem:#1}}

\newcommand{\sectn}[1]{Section~\ref{sec:#1}}

\newcommand{\sect}[1]{\ref{sec:#1}}

\newcommand{\pend}{\hfill \thicklines \framebox(6.6,6.6)[l]{}}

\newenvironment{proof*}[1]{\noindent {\sc  #1} \rm}{\pend}
\newenvironment{keyword}[1]{ {\bf #1} \rm}{}

\newtheorem{theorem}{Theorem}[section]
\newtheorem{lemma}{Lemma}[section]

\newtheorem{remark}{Remark}[section]

\newtheorem{corollary}{Corollary}[section]
\newtheorem{definition}{Definition}[section]

\newenvironment{mylist}[1]{\begin{list}{}
{\setlength{\itemindent}{#1mm}}
{\setlength{\itemsep}{0ex plus 0.2ex}}
{\setlength{\parsep}{0.5ex plus 0.2ex}}
{\setlength{\labelwidth}{10mm}}
}{\end{list}}

\newcommand{\setnewcounter} {
\setcounter{subsection}{0}
\setcounter{equation}{0}
\setcounter{conjecture}{0}
\setcounter{assumption}{0}
\setcounter{question}{0}
\setcounter{definition}{0}
\setcounter{theorem}{0}
\setcounter{corollary}{0}
\setcounter{lemma}{0}
\setcounter{proposition}{0}
\setcounter{remark}{0}
}

\begin{document}
 \title{\Large \bf Palm problems arising in BAR approach\\ and its applications}

\author{Masakiyo Miyazawa\footnotemark}
\date{\today}

\maketitle

\begin{abstract}

We consider Palm distributions arising in a Markov process with time homogeneous transitions which is jointly stationary with multiple point processes. Motivated by a BAR approach studied in the recent paper \cite{BravDaiMiya2023}, we are interested in two problems; when this Markov process inherits the same Markov structure under the Palm distributions, and how the state changes at counting instants of the point processes can be handled to derive stationary equations when there are simultaneous counts and each of them influences the state changes. We affirmatively answer the first problem, and propose a framework for resolving the second problem, which is applicable to a general stationary process, which is not needed to be Markov. We also discuss how those results can be applied in deriving BAR's for the diffusion approximation of queueing models in heavy traffic. In particular, as their new application, the heavy traffic limit of the stationary distribution is derived for a single server queue with a finite waiting room.
\end{abstract}

\keyword{Keywords:}{Palm distribution, Markov process, point process, basic adjoint relationship, stationary distribution, heavy traffic approximation, generalized Jackson network, finite queue}

\footnotetext[1]{Department of Information Sciences,
Tokyo University of Science, Noda, Chiba, Japan.}

\section{Introduction}
\label{sec:introduction}
\setnewcounter

We mainly consider a Markov process with time homogeneous transitions which is jointly stationary with multiple point processes. If these point processes have finite intensities, we can define Palm distributions concerning them (e.g., see \cite{BaccBrem2003,Miya2010}). In the recent paper \cite{BravDaiMiya2023}, these Palm distributions are used as one of the key tools for deriving a diffusion approximation in heavy traffic for the stationary distribution of a multi-class queueing network with static buffer priorities (SBP) which has generally distributed exogenous inter-arrival and service times. In particular, the Palm distributions are used to derive stationary equations, which are called basic adjoint relationship, BAR for short, whose asymptotic versions are called asymptotic BAR's. Those BAR's enable the diffusion approximation of queueing networks in heavy traffic. This method coins a BAR approach. In those derivations of BAR's, the following two facts are used; (1st) the dynamics at the jump instants of the process is unchanged under Palm distributions, and (2nd) simultaneous exogenous-arrivals and/or service completions do not influence the BAR. (1st) is proved for a marginal type of the Palm distribution in \cite{BravDaiMiya2023}, while, for (2nd), it is suggested to decompose the state change due to those simultaneous events into a sequence of intermediate states in \cite{BravDaiMiya2017,BravDaiMiya2023} (see also the detailed Palm distribution in \cite{Miya2010}) .

These facts motivate us to consider them in a more general context than those studied in \cite{BravDaiMiya2023}, that is, for a general Markov process jointly stationary with multiple point processes in which all the jump instants of the Markov process are counted by those point processes. In this general framework, we have two questions; (Q1) what kind of Markov properties are inherited when the probability law is changed to the Palm distribution, and (Q2) how can the state changes of the process be handled in deriving BAR's when point processes have simultaneous counts and each of them change the state of the Markov process ? We refer to (Q1) and (Q2) as Palm problems. 

We affirmatively answer (Q1) by \thr{Palm 1}, and propose a framework for resolving problems in (Q2). Namely, for (Q1), we show that, under the Palm distributions, the strong Markov property is preserved with the same transition operator while the jump kernel at state changes at counting instants of the point processes is unchanged if the counting instants are predictable stopping times. For (Q2), we show the proposed framework works well, for which the process is not needed to be Markov, namely, it is applicable to a general stationary process. We then discuss how those results can be applied to derive the BAR's.

We consider two examples to see how the solutions for (Q1) and (Q2) work. The first example is a generalized Jackson network, studied in \cite{BravDaiMiya2017}. We derive the BAR for this network under the proposed framework. In this example, we focus on how (Q2) is resolved. The second example is a single server queue with a finite waiting room. We derive the limit of its stationary distributions in diffusion scaling under heavy traffic conditions. This limit is identical with that of the corresponding Brownian model (\cite{Harr2013}). This example demonstrates the power of the BAR approach. A further application can be found in \cite{Miya2024}.

This paper makes up four sections. In \sectn{modeling}, we first give a framework for a general continuous-time process and point processes which are jointly stationary motivated by a queueing network, which answers (Q2). Then, we introduce a framework for a Markov process. In \sectn{main}, (Q1) is answered, and the BAR is derived under the proposed framework for a piecewise deterministic Markov process. Finally, in \sectn{application}, the BAR approach is applied to a generalized Jackson network and a single server queue with a finite waiting room.

\section{Modeling assumptions and BAR approach}
\label{sec:modeling}
\setnewcounter

In this section, we introduce a general framework for deriving the BAR. We start with an example, which motivates to propose the general framework.

\subsection{Motivated example}
\label{sec:motivated}

Consider a generalized Jackson network with $d$ stations, $GJ$-network for short. This is the queueing network which has $d$ service stations for positive integer $d$. Those stations are indexed by $1,2,\ldots,d$, and let $J_{d} = \{1,2,\ldots.d\}$. Let $K_{e}$ be the index set of the stations which have exogenous arrivals. $K_{e}$ is a subset of $J_{d}$, but need not be $J_{d}$. Namely, there may be stations which do not have exogenous arrivals.  At each station, customers are served in the first-come first-served manner by a single server.

The inter-arrival times of exogenous customers are assumed to be $i.i.d.$, and the service times at each station do so. Furthermore, they are independent, and positive $a.s.$ (almost surely). After service completions at station $i$, customers are independently routed to station $i'$ with probability $p_{i,i'}$ or leave the network otherwise. Let $P \equiv \{p_{i, i'}; i, i'\in J_{d}\}$, which is called a routing matrix. For this network to be a non-explosive open network, we assume that $(I-P)^{-1} < \infty$ for the routing matrix $P$.

As is well known, this $GJ$-network can be described by a continuous-time Markov process. For this description, let $L_{i}(t)$ be the queue length (including customers being served) at station $i \in J_{d}$ at time $t \ge 0$, and define a queue length process of the network by $L(\cdot) \equiv \{(L_{1}(t), L_{2}(t),\ldots,L_{d}(t)); t \ge 0\}$ with state space $\dd{Z}_{+}^{d}$, where $\dd{Z}_{+} = \{0,1,\ldots\}$. For $t \ge 0$, let $R_{j}(t)$ be the remaining arrival time of an exogenous customer at station $j \in K_{e}$. Similarly, let $R_{d+j}(t)$ be the remaining service time at station $j \in J_{d}$. We here conventionally define $R_{d+j}(t) =$ the service time of the next customer (even if no customer there) whenever service is completed at station $j$ at time $t$. This convention does not change $L(\cdot)$, and is important for the heavy traffic approximations in \cite{BravDaiMiya2017,BravDaiMiya2023}. However, for a newly started service time to be independent of the past history of the system, we have to replace the remaining service time by a newly and independently sampled service time when the service starts. This is needed to well keep the Markov property of $X(\cdot)$ defined below under the convention, and only possible for the service times which are independent of everything else at each station when they start to be processed.

For notational convenience, we put $R_{j}(t) \equiv 0$ for $j \in J_{d} \setminus K_{e}$. Then, $R_{j}(t)$ is defined for all $j \in K \equiv \{1,2,\ldots,2d\}$. Let $R(t)$ be the $2d$-dimensional vector whose $j$-th entry is $R_{j}(t)$. Define continuous-time process $X(\cdot) \equiv \{X(t); t \ge 0\}$ by
\begin{align*}
  X(t) = (L(t),R(t)) \in S \equiv \dd{Z}_{+}^{d} \times \dd{R}_{+}^{2d}.
\end{align*}
where $\dd{R}_{+} = [0,\infty)$. It is easy to see that $X(\cdot)$ is a Markov process with state space $S$. As usual, we can assume without loss of generality that $X(t)$ is right continuous with left-had-limits at all $t \ge 0$ for each fixed sample. We denote its left-hand-limit at time $t$ by $X(t-)$. 

We next introduce counting processes. Let $K_{s} = \{d+1,d+2,\ldots,2d\}$, and let $K_{+} = K_{e} \cup K_{s}$. Taking the convention on the remaining service times into account, we define counting process $N_{j}(\cdot) \equiv \{N_{j}(t); t \ge 0\}$ for $j \in K$ by
\begin{align}
\label{eq:Nj}
  N_{j}(t) = \sum_{s \in (0,t]} 1(R_{j}(s) > R_{j}(s-) = 0), \qquad & j \in K, t \ge 0,   
\end{align}
where $1(\cdot)$ is the indicator function of proposition ``$\cdot$''. Note that $N_{j}(t) \equiv 0$ for $j \in K \setminus K_{+} = J_{d} \setminus K_{e}$, which is called null, and that $N_{j}(t)$ for $j \in K_{+}$ is finite for each $t \ge 0$, that is, $N_{j}(\cdot)$ is locally finite, as long as finitely many customers arrive at station $j$ in a finite time interval. Furthermore, $N_{j}(t)$ is right continuous because $R_{j}(t)$ is right continuous. Obviously, all the discontinuous instants of $X(\cdot)$ are counted by some $N_{j}(\cdot)$.

A counting process can be viewed as a point process, which is a nonnegative integer-valued random measure on $(\dd{R}_{+},\sr{B}(\dd{R}_{+}))$ in which $\sr{B}(\dd{R}_{+})$ is the Borel field on $\dd{R}_{+}$. For the counting process $N_{j}(\cdot)$, define the point process $N_{j}$ through
\begin{align}
\label{eq:correspond 1}
  N_{j}((0,t]) = N_{j}(t), \qquad t > 0,
\end{align}
where $N_{j}(B)$ is the measure of $B \in \sr{B}(\dd{R}_{+})$. Obviously, the counting process $N_{j}(\cdot)$ is also determined by point process $N_{j}$, which is also called null for $j \in K \setminus K_{+}$.

Thus, the point process $N_{j}$ and counting process $N_{j}(\cdot)$ are different expressions of the same stochastic process. It will be advantageous for us to use both of these expressions. Note that the sample paths of $N_{j}(\cdot)$ are right continuous with left-hand-limits because $X(\cdot)$ has those properties. Note that $N_{j}$ is a simple point process, that is, $\Delta N_{j}(t) \equiv N_{j}(t) - N_{j}(t-) \le 1$ for $t \ge 0$ because the counting intervals are almost surely positive.

Then, we can construct a stochastic basis $(\Omega,\sr{F},\dd{F},\dd{P})$, where $\dd{F} \equiv \{\sr{F}_{t}; t \ge 0\}$ is a filtration on $\Omega$, and a shift operator semi-group $\Theta_{\bullet} \equiv \{\Theta_{t}; t \ge 0\}$ on it, namely, $\Theta_{0}(\omega) = \omega$ and $\Theta_{s+t}(\omega) = \Theta_{s} \circ \Theta_{t}(\omega) \equiv \Theta_{s}(\Theta_{t}(\omega))$ for $\omega \in \Omega$ and $s,t \ge 0$, such that
\begin{enumerate}[({2.}a)]
\item $X(\cdot)$ and $N_{j}(\cdot)$ for $j \in K$ are $\dd{F}$-adapted, that is, $X(t)$ and $N_{j}(t)$ are $\sr{F}_{t}$-measurable for $t \ge 0$, and their sample paths are right continuous with left-had-limits.
\item $\Theta_{\bullet}$ is $\sr{B}(\dd{R}_{+} \otimes \sr{F})/\sr{F}$ measurable, that is, $\{(t,\omega) \in \dd{R}_{+} \times \Omega; \Theta_{t}(\omega) \in A\} \in \sr{B}(\dd{R}_{+}) \otimes \sr{F}$ for each $A \in \sr{F}$, where
\begin{align*}
  \sr{B}(\dd{R}_{+}) \otimes \sr{F} = \sigma(\{B \times A; B \in \sr{B}(\dd{R}_{+}), A \in \sr{F}\}),
\end{align*}
in which $\sigma(\sr{A})$ denotes the minimal $\sigma$-field containing $\sr{A}$,
\item $\dd{F}$, $X(\cdot)$ and $N_{j}(\cdot)$ for $j \in K$ are consistent with $\Theta_{\bullet}$, namely,
\begin{align}
\label{eq:measurable-shift 1}
 & A \in \sr{F}_{t} \; \mbox{ implies } \; \Theta_{s}^{-1} A \in \sr{F}_{s+t}, \qquad s,t \ge 0,\\
\label{eq:X-consistent 1}
 & X(t) \circ \Theta_{s}(\omega) = X(s+t)(\omega), \qquad s,t \ge 0, \omega \in \Omega,\\
\label{eq:N-consistent 1}
 & N_{j}(t) \circ \Theta_{s}(\omega) = N_{j}(s+t)(\omega) - N_{j}(s)(\omega), \qquad s,t \ge 0, \omega \in \Omega,
\end{align}
where, for random variable $W$ and $A \in \sr{F}$,
\begin{align*}
  W \circ \Theta_{s}(\omega) = W (\Theta_{s}(\omega)), \qquad \Theta_{s}^{-1} A = \{\omega \in \Omega; \Theta_{s}(\omega) \in A\}.
\end{align*}
\end{enumerate}

Let $T_{e,i}$ for $i \in K_{e}$ and $T_{s,j}$ for $j \in J_{d}$, be random variables which represent the inter-arrival and service times at stations $i$ and $j$, respectively. Let $\vc{\alpha} \equiv (\alpha_{1}, \alpha_{2}, \ldots, \alpha_{d})$ be the solution of the traffic equation:
\begin{align}
\label{eq:GJN-traffic 1}
  \alpha_{j} = \lambda_{j} + \sum_{i=1}^{d} \alpha_{i} p_{i,j}, \qquad j \in J_{d}.
\end{align}
Because of the condition $(I-P)^{-1} < \infty$, $\vc{\alpha}$ uniquely exists and is finite. Define the traffic intensities $\rho_{i} \equiv \alpha_{i}/\mu_{i}$, where $\lambda_{i} = 1/\dd{E}(T_{e,i})$ for $i \in K_{e}$, $\lambda_{i} = 0$ for $i \in J_{d} \setminus K_{e}$, and $\mu_{i} = 1/\dd{E}(T_{s,i})$ for $i \in J_{d}$. Then, it is well known for the $GJ$-network that $X(\cdot)$ has a unique stationary distribution if \begin{enumerate}[(2.d)]
\item $\rho_{i} < 1$ for all $i \in J_{d}$ and the distribution of $T_{e,i}$ is spread out for all for $i \in K_{e}$.
\end{enumerate}
 If we take this stationary distribution as the distribution of the initial state $X(0)$, then Markov process $X(\cdot)$ and point process $N_{j}$ for $j \in K$ are jointly stationary processes, where a point process is said to be stationary if its distribution is unchanged by time shift. Since \eq{N-consistent 1} implies
\begin{align*}
  N_{j}(B) \circ \Theta_{s}(\omega) = N_{j}(B+s), \qquad B \in \sr{B}(\dd{R}_{+}),
\end{align*}
where $B+s = \{u \ge s; u-s \in B\}$. $N_{j}$ is indeed a stationary point process under the probability measure for which $X(\cdot)$ is a stationary process. This is the typical framework on which we work in this paper.

\subsection{Framework for stationary processes}
\label{sec:stationary}

We aim to consider a Markov process, but (Q2) can be answered for more general processes. Because of this, we start with a general continuous-time process $X(\cdot) \equiv \{X(t); t \ge  0\}$ with state space $S$ and general point processes $N_{j}$ for $j \in {K} \equiv \{1,2,\ldots,k\}$, where $S$ is a $d+k$ dimensional real vector space with Euclidean norm for positive integers $d$ and $k$. Denote the Borel $\sigma$-field on $S$ by $\sr{B}(S)$. The index set of the non-null point processes is denoted by $K_{+} \subset K$. Here, $K$ can be reduced to $K_{+}$, but it is convenient to keep both of them for describing the dynamics of $X(\cdot)$. Let $N_{j}(\cdot) \equiv \{N_{j}(t); t \ge 0\}$ be the counting process determined by point process $N_{j}$ through \eq{correspond 1}.

Taking the framework for the $GJ$-network process into account, we further assume that there is a stochastic basis $(\Omega,\sr{F},\dd{F},\dd{P})$ and a shift operator semi-group $\Theta_{\bullet} \equiv \{\Theta_{t}; t \ge 0\}$ satisfying the following conditions.
\begin{mylist}{3}
\item [(M1)] The conditions (2.a), (2.b) and (2.c) of \sectn{motivated} hold.
\item [(M2)] For each $j \in K_{+}$, $N_{j}$ is a locally finite simple point process, and, for $j \in K \setminus K_{+}$,  $N_{j}$ is a null point process.
\end{mylist}

For the counting processes $N_{j}(\cdot)$'s, define $N_{\rm{all}}(\cdot) \equiv \{N_{\rm{all}}(t); t \ge 0\}$ by
\begin{align*}
  N_{\rm{all}}(t) = \sum_{j \in {K}} N_{j}(t), \qquad t \ge 0.
\end{align*}
Let $N_{\rm{all}}$ be the point process determined by the counting process $N_{\rm{all}}(\cdot)$. Note that $N_{\rm{all}}$ may not be simple. We next generate a simple point process $N_{0}$ from $N_{\rm{all}}$. To this end, let $t_{0} = 0$, and let $t_{n}$ be the $n$-th jump instant of $N_{\all}(\cdot)$ for $n \ge 1$, then define the counting process $N_{0}(\cdot) \equiv \{N_{0}(t); t \ge 0\}$ by
\begin{align*}
  N_{0}(t) = \sum_{n=1}^{\infty} 1(t_{n} \le t), \qquad t \ge 0.
\end{align*}
and let $N_{0}$ be the point process determined by the counting process $N_{0}(\cdot)$. Obviously, $N_{0}$ is a simple point process.

\begin{mylist}{3}
\item [(M3)] All the discontinuous instants of $X(\cdot)$ are counted by $N_{0}(\cdot)$.
\end{mylist}

To answer (Q2), we next consider the state changes at the discontinuous instants $t_{n}$'s of $X(\cdot)$ caused by the simultaneous counts of the $N_{j}(\cdot)$'s. Since $t_{n}$'s are stopping times, we need a $\sigma$-fields up to a stopping time to describe such state changes. For this, we introduce the following two $\sigma$-fields for a stopping time $\tau$.
\begin{align*}
 & \sr{F}_{\tau} = \{A \in \sr{F}; A \cap \{\tau \le t\} \in \sr{F}_{t}, \forall t \ge 0\},\\
 & \sr{F}_{\tau-} = \sigma(\sr{F}_{0} \cup \{A \cap \{t < \tau\}; A \in \sr{F}_{t}, t \ge 0\}).
\end{align*}

We are now ready to propose the following decomposition for the state changes of $X(\cdot)$ at instants $t_{n}$'s, which is answer to (Q2).
\begin{mylist}{3}
\item [(M4)] There are $S$-valued random variables $Y_{j,n}$ for $j \in \{0\} \cup K$ and  $n \ge 0$ satisfying the following three conditions, where recall that $t_{0} = 0$: For $n \ge 0$,
\end{mylist}
\begin{mylist}{6}
\item [(M4.a)] $Y_{0,n} = X(t_{n}-)$, $Y_{k,n} = X(t_{n})$, and, for $j \in K$, $Y_{j,n} = Y_{j-1,n}$ if $\Delta N_{j}(t_{n}) = 0$.
\item [(M4.b)] For $j \in K$, $Y_{j,n}$ is a $\sr{F}_{t_{n}}$-measurable function of $Y_{j-1,n}$.
\item [(M4.c)] $Y_{j,n}$ is consistent with $\{\Theta_{t_{n}}; n \ge 0\}$, that is, $Y_{j,n} \circ \Theta_{t_{m}}= Y_{j,m+n}$ for $m \ge 0$.
\end{mylist}

We call $Y_{j,n}$ an intermediate state for $j \in K \setminus \{k\}$. We will see how (M4) works well in \lem{Palm 4}. Note that $X(t_{n}-)$ is $\sr{F}_{t_{n}-}$-measurable. This $\sigma$-field has a important role to describe the state change  from $X(t_{n}-)$ to $X(t_{n})$ in Sections \sect{Markov} and \sect{main}. This is the reason why not only $\sr{F}_{\tau}$ but also $\sr{F}_{\tau-}$ are introduced.

\begin{remark}\rm
\label{rem:M4}
In some queueing applications, assigning indexes to the point processes may be important for (M4) to be satisfied. For example, consider the single server queue in which customers are served autonomously in batches, then the queue length may be differently changed depending on the order of the arrival and service completion when they occur simultaneously. However, this is not the case for our applications because (M4) is satisfied for any their order (see \sectn{GJ-network}).
\pend
\end{remark}

We now make a key distributional assumption:
\begin{mylist}{3}
\item [(D1)] The time-shift operator $\{\Theta_{t}; t \ge 0\}$ on $(\Omega,\sr{F},\ol{\dd{F}},\dd{P})$  satisfies
\begin{align}
\label{eq:stationary 1}
 & \dd{P}(A) = \dd{P}(\Theta_{t}^{-1} A), \qquad A \in \sr{F}_{\infty}, t \ge 0.
\end{align}
\end{mylist}

\begin{remark}\rm
\label{rem:X0-}
The assumption (D1) is essentially equivalent to assuming that $X(\cdot)$ and $N_{j}$ for $j \in {K}$ are jointly stationary. From  this stationarity, $X(\cdot)$ can start at any time $s < 0$. To see this, for fixed $s < 0$, define $X_{s}(t) = X(t-s)$ for $t \ge s$, then
\begin{align*}
  X_{s}(t) \circ \Theta_{u} = X(t-s) \circ \Theta_{u} = X(t+u-s) = X_{s}(t+u), \qquad t \ge s, u \ge 0,
\end{align*}
and the probability laws of $\{X_{s}(t); t \ge s\}$ and $\{X_{s}(t); t \ge 0\}$ are identical with that of $X(\cdot)$. Hence, we can shift the starting time of $X(\cdot)$ from $0$ to $s < 0$. Thus, $X(0-)$ is well defined, and
\begin{align}
\label{eq:X0-}
  X(0-) \circ \Theta_{t_{m}} = X(t_{m}-), \qquad m \ge 1.
\end{align}
This $X(0-)$ will be used under Palm distributions.
\pend
\end{remark}

As for the point processes $N_{j}$ for $j \in K$, it is a stationary point process by (D1). We denote its intensity $\dd{E}[N_{j}(1)]$ by $\alpha_{j}$, then $\alpha_{j} = 0$ for $j \in K \setminus K_{+}$. For the non-null point processes, we assume:
\begin{mylist}{3}
\item [(D2)] $\quad 0 < \alpha_{j} < \infty$ for $j \in K_{+}$.
\end{mylist}

For stationary point process $N_{j}$ for $j \in K_{+}$, we define Palm distributions on $(\Omega,\sr{F})$. Let $t_{j,m}$ be the $m$-th be the counting time of $N_{j}(\cdot)$ for $j \in K_{+}$. By the definition of $N_{0}$, there is a unique $n$ such that $t_{j,m} = t_{n}$. From the stationary assumption, $0 < t_{j,m}< \infty$ for $m \ge 1$ and $t_{j,m} \uparrow \infty$ as $m \to \infty$. Since $N_{j}(\cdot)$ is adapted to $\dd{F}$, $t_{j,m}$ is a stopping time with respect to the filtration $\dd{F}$. 
Define the Palm distribution $\dd{P}_{j}$ on $(\Omega,\sr{F})$ for $j \in K_{+}$ as
\begin{align}
\label{eq:Palm 1}
  \dd{P}_{j}(A) = \alpha_{j}^{-1} \dd{E}\left[\sum_{m=1}^{\infty} 1_{A} \circ \Theta_{t_{j,m}} 1(t_{j,m} \le 1) \right], \qquad A \in \sr{F}_{\infty}.
\end{align}
It is easy to see that $\dd{P}_{j}(\Delta N_{j}(0) = 1) = 1$ because $\Theta_{t_{j,m}}^{-1} \{\Delta N_{j}(0) =1\} = \{\Delta N_{j}(t_{j,m})=1\}$. 

Recall the simple point process $N_{0}$ which has the same counting time as $N_{\rm{all}}$. Let $\alpha_{0} = \dd{E}[N_{0}(1)]$, then $\alpha_{0} \le \dd{E}[N_\all(1)] = \sum_{j \in K_{+}} \alpha_{j}$. Similarly to \eq{Palm 1} we define the Palm distribution concerning $N_{0}$ as
\begin{align}
\label{eq:Palm 2}
  \dd{P}_{0}(A) = \alpha_{0}^{-1} \dd{E}\left[\sum_{n=1}^{\infty} 1_{A} \circ \Theta_{t_{n}} 1(t_{n} \le 1) \right], \qquad A \in \sr{F}.
\end{align}

Note that this Palm distribution may be different from that of $N_{\rm{all}}$. Namely,
\begin{align}
\label{eq:Palm standard}
  \dd{P}_{\rm{all}}(A) = \alpha_\all^{-1} \dd{E}\left[\int_{0}^{1}  1_{A} \circ \Theta_{t} N_{\rm{all}}(dt)\right], \qquad A \in \sr{F}.
\end{align}

Define $\Delta h(t) = h(t) - h(t-)$ for the function $h$ from $\dd{R}_{+}$ to $\dd{R}$ which is right continuous with left-hand-limits. Let $f(X)(t) = f(X(t))$, then $\Delta f(X)(t) = f(X(t)) - f(X(t-))$, which is well defined for $t=0$ by \rem{X0-}. Define 
\begin{align*}
  \Delta_{j} f(X)(t_{n}) = f(Y_{j,n}) - f(Y_{j-1,n}), \qquad j \in K, n \ge 1, f \in C_{b}(S).
\end{align*}

We note the following fact, which will be used together with \thr{Palm 1} in applications.
\begin{lemma}\rm
\label{lem:Palm 4}
Under the assumptions (M1)--(M4), (D1) and (D2),
\begin{align}
\label{eq:Palm 4}
& \dd{E}_{0}\left[\Delta_{j} f(X)(0) | \Delta N_{j}(0) = 1\right] = \dd{E}_{j} \left[ \Delta_{j} f(X)(0)\right], \qquad j \in K_{+},\\
\label{eq:Palm 5}
& \alpha_{0} \dd{E}_{0}\left[\Delta f(X)(0) \right] = \sum_{j \in K_{+}} \alpha_{j} \dd{E}_{j} \left[ \Delta_{j} f(X)(0)\right],
\end{align}
where $\dd{E}_{j}$ represents the expectation under $\dd{P}_{j}$ for $j \in \{0\} \cup K_{+}$. 
\end{lemma}

\begin{remark}\rm
\label{rem:Palm 4}
Equation \eq{Palm 5} is obvious if $N_{j}$'s have no common point. However, if they have common points, then it is not immediate, and requires a proof.
\end{remark}

\begin{proof}
We first prove \eq{Palm 4}. Since, for $j \in K_{+}$,
\begin{align*}
 \Delta_{j} f(X)(t_{n}) 1(\Delta N_{j}(t_{n}) = 1) & = \sum_{m=1}^{\infty} (f(Y_{j,n}) - f(Y_{j-1,n})) 1(t_{n} = t_{j,m})\\
 & = \sum_{m=1}^{\infty} \Delta_{j} f(X)(t_{j,m}) 1(t_{n} = t_{j,m}),
\end{align*}
we have
\begin{align}
\label{eq:Palm 6}
 & \sum_{n=1}^{\infty} \Delta_{j} f(X)(t_{n}) 1(\Delta N_{j}(t_{n}) = 1) 1(t_{n} \le 1) = \sum_{n=1}^{\infty} \sum_{m=1}^{\infty} \Delta_{j} f(X)(t_{j,m}) 1(t_{n} = t_{j,m}) 1(t_{j,m} \le 1)  \nonumber\\
 & = \sum_{m=1}^{\infty} \Delta_{j} f(X)(t_{j,m}) \sum_{n=1}^{\infty} 1(t_{n} = t_{j,m}) 1(t_{j,m} \le 1)  = \sum_{m=1}^{\infty} \Delta_{j} f(X)(t_{j,m}) 1(t_{j,m} \le 1).
\end{align}
Hence, taking the expectation of both sides of this equation and applying the definitions $\dd{P}_{0}$ and $\dd{P}_{j}$, we have
\begin{align}
\label{eq:Palm 7}
  & \alpha_{0} \dd{E}_{0}\left[\Delta_{j} f(X)(0) 1(\Delta N_{j}(0) = 1)\right] = \alpha_{j} \dd{E}_{j} \left[ \Delta_{j} f(X)(0)\right].
\end{align}
Similarly, $\alpha_{0} \dd{P}_{0}\left[\Delta N_{j}(0) = 1\right] = \alpha_{j}$, so
\begin{align*}
  \dd{E}_{0}\left[\Delta_{j} f(X)(0) | \Delta N_{j}(0) = 1\right] & = \dd{E}_{0}\left[\Delta_{j} f(X)(0) 1(\Delta N_{j}(0) = 1)\right] / \dd{P}_{0}\left[\Delta N_{j}(0) = 1\right]\\
  & = \frac {\alpha_{0}}{\alpha_{j}} \dd{E}_{0}\left[\Delta_{j} f(X)(0) 1(\Delta N_{j}(0) = 1)\right].
\end{align*}
Hence, \eq{Palm 7} proves \eq{Palm 4}. From \eq{Palm 6}, we have
\begin{align*}
  \sum_{n=1}^{\infty} \Delta f(X)(t_{n}) 1(t_{n} \le 1) & = \sum_{n=1}^{\infty} \sum_{j \in K_{+}} \Delta_{j} f(X)(t_{n}) 1(\Delta N_{j}(t_{n}) = 1) 1(t_{n} \le 1) \\
  & = \sum_{m=1}^{\infty} \sum_{j \in K_{+}} \Delta_{j} f(X)(t_{j,m}) 1(t_{j,m} \le 1).
\end{align*}
Taking the expectation of this equation proves \eq{Palm 5}, 
\end{proof}

\subsection{Framework for a Markov process}
\label{sec:Markov}

Similar to the $GJ$-network in \sectn{motivated}, the process $X(\cdot)$ is a Markov process in most applications. In this section, we first recall definitions on a Markov and strong Markov processes, and introduce notations for them.

\begin{definition}\rm
\label{dfn:Markov 1}
For stochastic process $X(\cdot)$ with state space $S$ which is adapted to filtration $\dd{F}$, define transition operators $\dd{T}_{\bullet} \equiv \{\dd{T}_{t}; t \ge 0\}$ as
\begin{align*}
 & \dd{T}_{t}f(x) = \dd{E}[f(X(s + t))|X(s)=x], \qquad s, t \ge 0, x \in S, f \in C_{b}(S),
\end{align*}
where $C_{b}(S)$ is the set of all bounded continuous functions from $[0,\infty)$ to $S$. Then, the $X(\cdot)$ is called a Markov process with respect to $\dd{F}$ if
\begin{align}
\label{eq:kernel 1}
 & \dd{E}[f(X(s+t))|\sr{F}_{s}] = \dd{T}_{t}f(X(s)) \; \mbox{ a.s. $\dd{P}$}, \qquad s,t \ge 0, f \in C_{b}(S),
\end{align}
and called a strong Markov process with respect to $\dd{F}$ if for every finite stopping time $\tau$, 
\begin{align}
\label{eq:kernel 2}
 & \dd{E}[f(X(\tau+t))|\sr{F}_{\tau}] = \dd{T}_{t}f(X(\tau)) \; \mbox{ a.s. $\dd{P}$}, \qquad t \ge 0, f \in C_{b}(S).
\end{align}
\end{definition}

We next define kernels for the state transitions of $X(\cdot)$ at the counting instants of $N_{0}(\cdot)$ under the assumptions (M1) and (M2).
\begin{definition}\rm
\label{dfn:jump kernel 1}
Define jump trasition kernels $\dd{Q}_{\bullet} = \{\dd{Q}_{n}; n \ge 0\}$ with respect to $N_{0}$ by
\begin{align}
\label{eq:transition operator}
 & \dd{Q}_{n}f(x) = \left\{
\begin{array}{ll}
 \dd{E}[f(X(0))|X(0-)=x], & n=0, \\
 \dd{E}[f(X(t_{n}))|X(t_{n}-)=x], \; & n \ge 1,
\end{array}
\right. \qquad x \in S, f \in C_{b}(S),
\end{align}
where recall that $t_{n}$ is the $n$-th counting time of $N_{0}(\cdot)$.
\end{definition}

The jump kernel $\dd{Q}_{\bullet}$ has a similar conditional property to \eq{kernel 2} of $\dd{T}_{\bullet}$, as shown below.
\begin{lemma}\rm
\label{lem:jump kernel 1}
Assume (M1) and (M2). If $X(\cdot)$ is a strong Markov process with respect to $\dd{F}$ and if $\tau$ is finite and predictable stopping time, then
\begin{align}
\label{eq:kernel 3}
 & \dd{E}[f(X(\tau))|\sr{F}_{\tau-}] = \dd{E}[f(X(\tau))|X(\tau-)] \; \mbox{ a.s. $\dd{P}$}, \qquad t > 0, f \in C_{b}(S).
\end{align}
In particular, if $t_{n}$'s are predictable, then
\begin{align}
\label{eq:jump kernel 1}
 & \dd{E}[f(X(t_{n}))|\sr{F}_{t_{n}-}] = \dd{Q}_{n}f(X(t_{n}-)) \; \mbox{ a.s. $\dd{P}$}, \qquad n \ge 1, f \in C_{b}(S).
\end{align}
\end{lemma}

\begin{proof}
We refer to the following fact in Theorem 3 in Section 2.4 of \cite{ChunWals2005}.
\begin{align}
\label{eq:kernel 4}
 & \dd{E}[f(X(\tau+t))|\sr{F}_{\tau-}] = \dd{E}[f(X(\tau+t))|X(\tau-)) \; \mbox{ a.s. $\dd{P}$}, \qquad t > 0, f \in C_{b}(S),
\end{align}
where Feller property is not needed because $t > 0$. Taking the limits of both sides of \eq{kernel 4} for $t \downarrow 0$, we have \eq{kernel 3} by the right-continuity of $X(\cdot)$. It remains to prove \eq{jump kernel 1}, but this is immediate from \eq{kernel 3} because $t_{n}$'s are finite stopping times by (M2) and the definition of $N_{0}$.
\end{proof}

Both of $\dd{T}_{\bullet}$ and $\dd{Q}_{\bullet}$ are important for our stochastic analysis, so we focus on them to answer question (Q1).

\section{Main results}
\label{sec:main}
\setnewcounter

We now answer the questions (Q1).

\begin{theorem}\rm
\label{thr:Palm 1}
For positive integrers $d$ and $k$, let $X(\cdot)$ and $N_{j}$ for $j \in K \equiv \{1,2,\ldots,k\}$ be a process with state space $S \equiv \dd{R}^{d+k}$ and point processes, respectively, on $(\Omega,\sr{F},\dd{F},\dd{P})$ such that they satisfy the assumptions (M1)--(M3), (D1) and (D2) and are jointly stationary. If $X(\cdot)$ is a strong Markov process with respect to $\dd{F}$ which has time homogeneous transition kernel $\dd{T}_{\bullet}$, then $X(\cdot)$ is also a strong Markov process with the same transition kernel $\dd{T}_{\bullet}$ under $\dd{P}_{0}$. That is, for any finite stopping time $\tau$,
\begin{align}
\label{eq:Markov 1}
  & \dd{E}_{0}[f(X(\tau+t))|\sr{F}_{\tau}] = \dd{T}_{t}f(X(\tau)), \mbox{ a.s. $\dd{P}_{0}$}, \qquad f \in C_{b}(S), t \ge 0.
\end{align}
Furthermore, if $t_{n}$'s are predictable stopping times for $n \ge 1$ and $j \in K_{+}$, then, for $f \in C_{b}(S)$,
\begin{align}
\label{eq:jump 2}
 & \dd{E}_{0}(f(X(t_{n}))|\sr{F}_{t_{n}-}) = \dd{Q}_{n}f(X(t_{n}-)), \quad a.s.\; \dd{P}_{0}, \qquad n \ge 0, j \in K_{+},
\end{align}
where recall $t_{0} = 0$, and $\dd{E}_{0}$ stands for the expectation under the Palm distribution $\dd{P}_{0}$.
\end{theorem}

\begin{remark}\rm
\label{rem:Palm 1}
Not strong but just Markov property \eq{kernel 1} is also unchanged under $\dd{P}_{0}$, which is proved, replacing $\tau$ by a positive constant in the proof below.
\pend
\end{remark}

\begin{proof}
Since \eq{Markov 1} is equivalent to
\begin{align}
\label{eq:Markov 2}
 & \dd{E}_{0}[f(X(\tau+t)) 1_{A}] = \dd{E}_{0}\left[\dd{T}_{t}f(\tau) 1_{A}\right], \quad f \in C_{b}(S), A \in \sr{F}_{\tau}, t \ge 0,
\end{align}
\eq{Markov 1} is obtained from \eq{Palm 2} if we prove
\begin{align}
\label{eq:Markov 3}
 & \dd{E}\left[(f(X(\tau+t)) 1_{A}) \circ \Theta_{t_{n}} 1(t_{n} \le 1)\right]  = \dd{E}\left[(\dd{T}_{t}f(\tau) 1_{A}) \circ \Theta_{t_{n}} 1(t_{n} \le 1)\right], \quad n \ge 1.
\end{align}
Since $X(t) \circ \Theta_{t_{n}}(\omega) = X(t) (\Theta_{t_{n}(\omega)}(\omega)) = X(t_{n}(\omega) + t)(\omega)$,
\begin{align*}
  X(\tau+t) \circ \Theta_{t_{n}}(\omega) & = X\left( t_{n}(\omega) + (\tau \circ \Theta_{t_{n}})(\omega) + t\right)(\omega).
\end{align*}
Hence, \eq{Markov 3} is equivalent to
\begin{align}
\label{eq:Markov 4}
& \dd{E}[f(X(t_{n}+\tau \circ \Theta_{t_{n}}+t)) (1_{A} \circ \Theta_{t_{n}}) 1(t_{n} \le 1)]  \nonumber\\
& \quad = \dd{E}[\dd{T}_{t}f(t_{n}+\tau \circ \Theta_{t_{n}}) (1_{A} \circ \Theta_{t_{n}}) 1(t_{n} \le 1)], \qquad n \ge 0.
\end{align}
Here, $t_{n} + \tau \circ \Theta_{t_{n}}$ is a finite stopping time by Theorem 11 in Section 1.3 of \cite{ChunWals2005}. Since $\Theta_{\bullet}$ is consistent with $\dd{F}$, $\Theta^{-1}_{t_{n}} A \in \sr{F}_{t_{n} + \tau \circ \Theta_{t_{n}}}$ for $A \in \sr{F}_{\tau}$. Hence, using the strong Markov property of $X(\cdot)$ under $\dd{P}$, the left-hand side of \eq{Markov 4} for $A \in \sr{F}_{\tau}$ becomes
\begin{align*}
  & \dd{E}[\dd{E}(f(X(t_{n} + \tau \circ \Theta_{t_{n}}+t) ) 1_{A} \circ \Theta_{t_{n}} 1(t_{n} \le 1)|\sr{F}_{t_{n}+\tau \circ \Theta_{t_{n}}})] \nonumber\\
  & \quad = \dd{E}[\dd{E}(f(X(t_{n} + \tau \circ \Theta_{t_{n}}+t) )|\sr{F}_{t_{n}+\tau \circ \Theta_{t_{n}}}) 1_{A} \circ \Theta_{t_{n}} 1(t_{n} \le 1)] \nonumber\\
  & \quad = \dd{E}[\dd{E}(f(X(t_{n} + \tau \circ \Theta_{t_{n}}+t) )|X(t_{n}+\tau \circ \Theta_{t_{n}})) 1_{A} \circ \Theta_{t_{n}} 1(t_{n} \le 1)] \nonumber\\
  & \quad = \dd{E}[\dd{T}_{t}f(t_{n}+\tau \circ \Theta_{t_{n}}) 1_{A} \circ \Theta_{t_{n}} 1(t_{n} \le 1)].
\end{align*}
Thus, \eq{Markov 3} is obtained, so \eq{Markov 1} is proved. As for jump kernel $\dd{Q}_{\bullet}$, for $n \ge 1$, \eq{jump 2} is immediate from \eq{jump kernel 1} because the transition structure of $X(\cdot)$ is unchanged under $\dd{P}_{0}$, which is just proved. For $n=0$, \eq{jump 2} follows from the definition of $\dd{P}_{0}$ and the fact that
\begin{align*}
  \dd{E}[(f(X(0)) 1_{A}) \circ \Theta_{t_{n}}] = \dd{E}[(\dd{Q}_{0}f(X(0-)) 1_{A}) \circ \Theta_{t_{n}}], \qquad n \ge 1, A \in \sr{F}_{0-}.
\end{align*}
Since $\Theta_{t_{n}}^{-1} A \in \sr{F}_{t_{n}-}$ for $A \in \sr{F}_{0-}$, this is equivalent to
\begin{align*}
  \dd{E}[f(X(t_{n})) 1_{A}] = \dd{E}[\dd{Q}_{n}f(X(t_{n}-)) 1_{A}], \qquad n \ge 1, A \in \sr{F}_{t_{n-}},
\end{align*}
which is immediate from \eq{jump kernel 1} of \lem{jump kernel 1}. Thus, the proof is completed.
\end{proof}

The following corollary is immediate from \eq{jump 2} for $n=0$ of \thr{Palm 1}.

\begin{corollary}\rm
\label{cor:jump kernel 2}
Under all the assumptions of \thr{Palm 1},
\begin{align}
\label{eq:jump kernel 2}
  \dd{E}_{0}[\Delta f(X)(0)] = \dd{E}_{0}[\dd{Q}_{0}f(X(0-))- f(X(0-))], \quad f \in C_{b}(S).
\end{align}
\end{corollary}

We consider the special case that the sample paths of Markov process $X(\cdot)$ are deterministically and continuously partially differentiable between adjacent discontinuous instants. We call this $X(\cdot)$ a piecewise deterministic Markov process, which is known to be a strong Markov process (e.g., see Section 25 of \cite{Davi1993}). Let $X_{i}(t)$ be the $i$-th entry of $X(t)$ for $i \in D \equiv \{1,2,\ldots,d+k\}$, and denote its derivative by $X'_{i}(t)$. Let $C^{p}(S)$ be the set of all continuously partial differentiable function from $S$ to $\dd{R}_{+}$. Define operator $\sr{H}$ on $C^{p}(S)$ as
\begin{align*}
  \sr{H}f(x) = \sum_{i \in D} h_{i}(x) \frac {\partial}{\partial x_{i}} f(x), \qquad f \in C^{p}(S), h_{i} \in C(S), i \in D,
\end{align*}
where $C(S)$ is the set of all continuous functions from $S$ to $\dd{R}$, and function $h_{i}$ is determined through $X_{i}'(t) = h_{i}(X(t))$. Then, by elementary calculus, we have
\begin{align}
\label{eq:evolution 1}
 & f(X(t)) - f(X(0))  \nonumber\\
 & \quad = \int_{0}^{t} \sr{H}f(X)(u) du + \sum_{j \in K_{+}} \sum_{n = 1}^{\infty} 1(t_{j,n} \le t) \Delta_{j} f(X)(t_{j,n}), \qquad t \ge 0.
\end{align}

Taking the expectation of \eq{evolution 1} and applying \lem{Palm 4}, we have the following lemma.
\begin{lemma}\rm
\label{lem:BAR 1}
For the piecewise deterministic Markov process $X(\cdot)$ and point processes $N_{j}$ for $j \in K$ satisfying the assumptions (M1)--(M4), (D1) and (D2),
\begin{align}
\label{eq:stationary 2}
  \dd{E}\left[\sr{H}f(X)(0)\right] + \sum_{j \in K_{+}} \alpha_{j} \dd{E}_{j}\left[\Delta_{j}f(X)(0)\right] = 0, \qquad f \in C_{b}^{p}(S),
\end{align}
where $C_{b}^{p}(S)$ is the set of all bounded functions in $C^{p}(S)$.
\end{lemma}
\begin{remark}\rm
\label{rem:BAR 1}
Neither strong Markov nor the predictability of $t_{n}$'s is needed for this lemma to hold. However, we do need both properties to evaluate $\dd{E}_{j}\left[\Delta_{j}f(X)(0)\right]$ using the dynamics of $X(\cdot)$ at time $0$ under $\dd{P}_{0}$. In our applications in \sectn{application}, they hold true because $X(\cdot)$ is piecewise Markov and $X(t-)$ has the information on $\Delta N_{0}(t)$ for any $t \ge 0$, which implies that $t_{n}$'s are predictable.
\pend 
\end{remark}

Equation \eq{stationary 2} is a stationary equation, and a special case of the rate conservation law, but different from the standard one which assumes that $N_{j}$'s do not have a common point. We refer to \eq{stationary 2} as a basic adjoint relationship, BAR for short. It is shown in \cite{Miya1991} that \eq{stationary 2} can be used to characterize the stationary distribution of $X(\cdot)$.

\section{Applications to queueing models}
\label{sec:application}
\setnewcounter

In this section, we apply the BAR of \lem{BAR 1} for two queueing models to see how the answers to Q(1) and Q(2) works,

\subsection{Generalized Jackson network}
\label{sec:GJ-network}

Let us consider the $GJ$-network of \sectn{motivated}. Let $T_{e,j,n}$ be the $n$-th inter arrival time of exogenous customers at station $j \in K_{e}$, let $T_{s,i,n}$ be the $n$-th service time at station $i$, and let $\Psi_{i}(n)$ be the station to which the $n$-th service completion customer at station $i$ is routed or vanish otherwise. It is assumed that $\{T_{e,i,n}; m \ge 1\}$ for $i \in K_{e}$, $\{T_{s,i,n}; n \ge 1\}$ and $\{\Psi_{i}(n); n \ge 1\}$ for $i \in J_{d} = \{1,2,\ldots,d\}$ are sequences of $i.i.d.$ positive random variables which are independent of everything else. Note that $\dd{P}(\Psi_{i}(n)=i') = p_{i, i'}$ for $i, i' \in J_{d}$. Recall that $N_{i}(\cdot)$ is the counting process of arrivals at station $i$ for $i \in K_{e}$ and null for $i \in J_{d} \setminus K_{e}$, while $N_{j}(\cdot)$ is the counting process of service completions at station $j-d$ for $j \in K_{s}$. Recall that $K_{s} = \{d+1,d+2,\ldots,2d\}$ and $K_{+} = K_{e} \cup K_{s}$.

In \sectn{motivated}, we described this network by $X(t) = (L(t), R(t)); t \ge 0\}$ whose state space is $S = \dd{Z}_{+}^{d} \times \dd{R}_{+}^{2d}$. It is easy to see that $X(\cdot)$ is a piecewise Markov process, so it is strong Markov as noted in \sectn{main}. We assume the stability condition (2.d), then we can consider $X(\cdot)$ to be a stationary process. Then, it can be shown that $N_{j}$ are stationary point processes with intensity $\alpha_{j}$ obtained from \eq{GJN-traffic 1}. Hence, (D1) and (D2) is satisfied. Furthermore, the $n$-th counting time of $N_{j}(\cdot)$ for $j \in K_{+}$, $t_{j,n}$ is predictable because $t_{j,n} = \inf \{t > t_{j,n-1}; R_{j}(t-) = 0\}$ for $n \ge 1$ by \eq{Nj}. Then, the $n$-th counting instant $t_{n}$ f $N_{0}(\cdot)$ is also predictable, which is the $n$-th jump instant of $X(\cdot)$. Since (M1)--(M3) are obviously satisfied, all the assumptions in \thr{Palm 1} holds.

For this $X(\cdot)$ and the counting processes $N_{j}$ for $K = \{1,2,\ldots,2d\}$, we define $\{Y_{j,n}; n \ge 1, j \in \{0\} \cup K\}$ in the following way. Let $\vcn{e}_{j}$ be the unit vector in $\dd{R}_{+}^{3d}$ whose $j$-th entry equals 1 while the other entries vanish. For $j \in J_{d}$, $n \ge 1$, define
\begin{align*}
  Y_{j,n} = \left\{
\begin{array}{ll}
 Y_{j-1,n} + \vcn{e}_{j} + T_{e,j,N_{j}(t_{n})} \vcn{e}_{d+j}, \quad & \Delta N_{j}(t_{n}) = 1 \mbox{ (obviously $j \in K_{e}$)},\\
 Y_{j-1,n}, & \mbox{otherwise},
\end{array}
\right.
\end{align*}
where $T_{e,j,N_{j}(t_{n})}$ is independent of $Y_{j-1,n}$ under $\dd{P}$. This $Y_{j,n}$ is the state changed from $Y_{j-1,n}$ by exogenous arrivals at station $j$ at time $t_{n}$. For $j \in K_{s}$, define
\begin{align*}
  Y_{j,n} = \left\{
\begin{array}{ll}
 \ds Y_{j-1,n} - \vcn{e}_{j-d} + \sum_{j' \in J_{d}} 1(\Psi_{j-d}(\ell)=j') \vcn{e}_{j'} + T_{s,j-d,N_{j}(t_{n})} \vcn{e}_{d+j}, & \Delta N_{j}(t_{n}) = 1,\\
 Y_{j-1,n}, & \mbox{otherwise},
\end{array}
\right.
\end{align*}
where $T_{s,j-d,N_{j}(t_{n})}$ is independent of $Y_{j-1,n}$. Similarly to the arrival case, this $Y_{j,n}$ is the state changed from $Y_{j-1,n}$ by service completion at station $j-d$ at time $t_{n}$.

It is not hard to see that $\{Y_{j,n}; j \in K, n \ge 0\}$ satisfy (M4), so $Y_{j,n}$'s are intermediate states. Note that these intermediate states describe the case that exogenous arrivals occurs before service completions. However, the state change from $X(t_{n}-)$ to $X(t_{n})$ is unchanged for any order of exogenous arrivals and service completions in this network model. Namely, (M4) is satisfied for any their order.

We next consider how the intermediate state changes are incorporated in computing a BAR. In \sectn{main}, $S = \dd{R}^{d+k}$, while $k=2d$ and the state space $S = \dd{Z}_{+} \times \dd{R}_{+}^{2d}$ for the $GJ$-network. However, $\dd{Z}_{+} \times \dd{R}_{+}^{2d}$ can be embedded into $\dd{R}^{d+k}$ as topological spaces, so all the results in \sectn{main} are available for the $GJ$-network. Since $X_{j}'(t) = L_{j}'(t) = 0$, $R'_{j}(t) = -1$ for $j \in K_{e}$ and $R'_{j}(t) = 1(L_{j-d}(t) \ge 1)$ for $j \in K_{s}$, we have
\begin{align*}
  \sr{H}f(X(t)) = - \sum_{i \in K_{e}} f_{d+i}(X(t))  - \sum_{i \in J_{d}} f_{2d+i}(X(t)) 1(L_{i}(t) \ge 1), \qquad f \in C^{1p}(S),
\end{align*}
where $C^{1p}(S)$ is the set of all functions from $S$ to $\dd{R}$ which is partially continuously differentiable. Obviously, $t_{j,n}$'s are predictable, and $f_{j}(\vc{x})$ is the partial derivative concerning the $j$-th entry $x_{j}$ of $\vc{x}$. Then, by \lem{BAR 1}, we have the BAR:
\begin{align}
\label{eq:GJN-BAR 1}
 & \dd{E} \Big[\sum_{j \in K_{e}} f_{d+j}(X(0))  + \sum_{j \in J_{d}} f_{2d+j}(X(0)) 1(L_{j}(0) \ge 1)\Big]  \nonumber\\
 & \qquad = \sum_{j \in K_{+}} \alpha_{j} \dd{E}_{j}\left[f(Y_{j,0}) - f(Y_{j-1,0})\right], \qquad f \in C_{b}^{p}(S),
\end{align}
where $\dd{E}_{j}$ stands for the expectation under Palm distribution $\dd{P}_{j}$ concerning $N_{j}$, and
\begin{align*}
  \dd{E}_{j}[f(Y_{j,0})] = \left\{
\begin{array}{ll}
 \dd{E}_{j}\left[f(Y_{j-1,0} + \vcn{e}_{j} + \vcn{e}_{d+j} T_{e,j,N_{j}(t_{n})})\right], & j \in K_{e},  \\
 \sum_{j' =1}^{d} p_{j,j'} \dd{E}_{j}\left[f(Y_{j-1,0} - \vcn{e}_{j-d} + \vcn{e}_{j'-d} + \vcn{e}_{d+j} T_{s,j-d,N_{j}(t_{n})})\right], \; & j \in K_{s}.  
\end{array}
\right.
\end{align*}
This is the way that we resolve (Q2).

To use \eq{GJN-BAR 1} for applications, it remains to evaluate the expectations under the Palm distributions $\dd{P}_{j}$ in \eq{GJN-BAR 1}. This is a hard problem. We resolve it by choosing a test function $f$ so that those expectations vanishes or is ignorable. In \cite{BravDaiMiya2017}, the following test function $f_{\vc{\theta},r}$ is chosen for parameter $\vc{\theta} \in \dd{R}_{-}^{d}$ and scaling factor $r \in (0,1]$, where $\dd{R}_{-}^{d} = (-\infty,0]$.
\begin{align}
\label{eq:GJN-test-f 1}
  f_{\vc{\theta},r}(\vc{x}) = \exp\left(\br{\vc{z}, \vc{\theta}} - \br{\vc{\eta}(\vc{\theta},r),\vc{y}_{e} \wedge 1/r} - \br{\vc{\zeta}(\vc{\theta},r),\vc{y}_{s} \wedge 1/r}\right),
\end{align}
where $\vc{x} = (\vc{z},\vc{y}_{e},\vc{y}_{s})$ for $\vc{z} \in \dd{Z}_{+}^{d}$ and $\vc{y}_{e}, \vc{y}_{s} \in \dd{R}_{+}^{d}$, and, for $\vc{\theta} = (\theta_{1}, \theta_{2}, \ldots, \theta_{d})$, $\vc{\eta}(\vc{\theta},r) = (\eta_{1}(\theta_{1},r), \eta_{2}(\theta_{2},r), \ldots, \eta_{d}(\theta_{d},r))$ and $\vc{\zeta}(\vc{\theta},r) = (\zeta_{1}(\vc{\theta},r), \zeta_{2}(\vc{\theta},r), \ldots, \zeta_{d}(\vc{\theta},r))$ which  are uniquely determined by 
\begin{align}
\label{eq:GJ-network boundary 1}
  & e^{\theta_{i}} \dd{E}\left[e^{-\eta_{i}(\theta_{i},r) (T_{e,i} \wedge 1/r)}\right] \; \mbox{ for } \; i \in K_{e}, \qquad \eta_{i}(\theta_{i},r) = 0 \; \mbox{ for } \; i \in J_{d} \setminus K_{e},\\
\label{eq:GJ-network boundary 2}
  & \sum_{i'=0}^{d} p_{i,i'} e^{-\theta_{i}+\theta_{i'}} \dd{E}\left[e^{-\zeta_{i}(\vc{\theta},r) (T_{s,i} \wedge 1/r)}\right] = 1, \quad i \in J_{d},
\end{align}
where $T_{e,j}$ and $T_{s,j}$ are random variables subject to the same distributions as $T_{e,j,n}$ and $T_{s,j,n}$, respectively, $p_{i,0} = 1 - \sum_{i'=1}^{d} p_{i,i'}$ and $\theta_{0} = 0$. Since the dynamics of $X(\cdot)$ under $\dd{P}$ is unchanged under $\dd{P}_{0}$ by \thr{Palm 1}, we have, by \eq{GJ-network boundary 1},
\begin{align*}
 & \dd{E}_{0}\left[\Delta_{j} f_{\vc{\theta},r}(X)(0) 1(\Delta N_{j}(0)=1)\right] = \dd{E}_{0}\left[(f_{\vc{\theta},r}(Y_{j,0}) - f_{\vc{\theta},r}(Y_{j-1,0})) 1(\Delta N_{j}(0)=1)\right]\\
 & \quad = \dd{E}_{0}\left[(f_{\vc{\theta},r}(Y_{j-1,0})[e^{\theta_{j} + \eta_{j}(\theta_{j},r) (T_{e,j,N_{j}(0)} \wedge 1/r)} - 1] 1(\Delta N_{j}(0)=1)\right]\\
 & \quad = \dd{E}_{0}\left[(f_{\vc{\theta},r}(Y_{j-1,0}) 1(\Delta N_{j}(0)=1)\right] \dd{E}\left[e^{\theta_{j} + \eta_{j}(\theta_{j},r) (T_{e,j} \wedge 1/r)} - 1\right] = 0.
\end{align*}
Hence, $\dd{E}_{j}\left[f(Y_{j,0}) - f(Y_{j-1,0})\right] = 0$ for $j \in K_{e}$ by \eq{Palm 4} of \lem{Palm 4}. Similarly, $\dd{E}_{j-d}\left[f(Y_{j,0}) - f(Y_{j-1,0})\right] = 0$ for $j \in K_{s}$. Thus, it follows from \eq{GJN-BAR 1} for $f = f_{r\vc{\theta},r}$ that
\begin{align}
\label{eq:GJN-BAR 2}
  \sum_{i \in E} \eta_{i}(r\theta_{i},r) \dd{E}\left[f_{r\vc{\theta},r}(X(0))\right] + \sum_{i \in J} \zeta_{i}(r\theta_{i},r) \dd{E}\left[f_{r\vc{\theta},r}(X(0)) 1(L_{i}(0) \ge 1)\right] = 0.
\end{align}
This is the BAR for deriving the weak limit of $r L^{(r)}$ as $r \downarrow 0$ in \cite{BravDaiMiya2017}, where $X^{(r)}\equiv (L^{(r)},R^{(r)})$ is the random vector subject to the stationary distribution of the $r$-th $GJ$-network which satisfies heavy traffic conditions as $r \downarrow 0$.

\subsection{GI/G/1 queue with a finite waiting room}
\label{sec:2L queue}

We next consider a single server queue which accepts customers only when its queue length including a customer in service is less than the threshold $\ell_{0} > 0$. In this model, customers arrive subject to a renewal process, and served in the FCFS manner. Their service times are independent and identically distributed. We refer to this queueing model as a $GI/G/1/\ell_{0}$ queue.

The $GI/G/1/\ell_{0}$ queue corresponds to a one-dimensional Brownian motion with two reflecting barriers, which is fully studied in \cite{Harr2013}. For example, the stationary distribution is derived in Proposition 6.6 of \cite{Harr2013}. We note that the $GI/G/1/\ell_{0}$ queue has slightly different behaviors on the lower and upper boundaries while those of the reflecting Brownian motion are basically the same reflecting structure. Nevertheless, it will be shown in \thr{FQ-limit 1} that the heavy traffic limit of the stationary distribution of the scaled queue size agrees with that of the reflecting Brownian motion.

This limiting distribution is closely related to that of a two node closed network which has a fixed number of customers, studied in \cite{HarrWillChen1990}, but it is different from ours because the arrival processes in a closed network are not renewal. 

We index the $GI/G/1/\ell_{0}$ queue by $r \in (0,1]$, which is also used as a scaling factor for its queue size. Since we only consider its stationary distribution, we do not scale time. For the $r$-th model, let $\ell^{(r)}_{0}$ be the threshold, and let $L^{(r)}(t)$ be the number of customers in the system at time $t$. Let $t^{(r)}_{e,n}$ be the $n$-th arrival time of a customer, and let $T^{(r)}_{s,n}$ be the service time of the $n$-th served customer. Let $T^{(r)}_{e,n} = t^{(r)}_{e,n} - t^{(r)}_{e,n-1}$, which is assumed to be independent of $\{t^{(r)}_{e,n'}; 0 \le n' \le n-1\}$, and has a common distribution. Denote a random variable subject to this distribution by $T^{(r)}_{e}$. Similarly, we assume that $\{T^{(r)}_{s,n}; n \ge 1\}$ is a sequence of $i.i.d.$ random variables, and denote a random variable subject to their common distribution by $T_{s}^{(r)}$.

We will use the following notations.
\begin{align*}
 & m^{(r)}_{e} = \dd{E}(T^{(r)}_{e}), \qquad \lambda^{(r)} = 1/m^{(r)}_{e}, \qquad \sigma^{(r)}_{e} = \dd{E}[(T^{(r)}_{e} - m^{(r)}_{e})^{2}],\\
 & m^{(r)}_{s} = \dd{E}(T^{(r)}_{s}), \qquad \mu^{(r)} = 1/m^{(r)}_{s}, \qquad \sigma^{(r)}_{s} = \dd{E}[(T^{(r)}_{s} - m^{(r)}_{s})^{2}],\\
 & \rho^{(r)} = \lambda^{(r)}/\mu^{(r)}, \qquad i = 1,2,
\end{align*}
where all of those quantities are assumed to be finite and positive.

For the $r$-th model, let $X^{(r)}(t) \equiv (L^{(r)}(t), R^{(r)}_{1}(t), R^{(r)}_{2}(t))$, and let $N^{(r)}_{1}(\cdot)$ and $N^{(r)}_{2}(\cdot)$ be the counting processes for the arrival and service completion instants of customers, respectively. Since we only consider a countable number of $r \in (0,1]$ such that $r \downarrow 0$, we can construct a stochastic basis $(\Omega,\sr{F},\dd{F},\dd{P})$ such that $X^{(r)}(\cdot)$, $N^{(r)}_{e}$ and $N^{(r)}_{s}$ are defined on a stochastic basis $(\Omega,\sr{F}^{(r)},\dd{F}^{(r)},\dd{P})$ satisfying $(\sr{F}^{(r)},\dd{F}^{(r)}) \subset (\sr{F},\dd{F})$ and the conditions (M1)--(M4) for all countable $r$'s. Obviously, $X^{(r)}(\cdot) \equiv \{X^{(r)}(t); t \ge 0\}$ is a strong Markov process with respect to $\dd{F}^{(r)}$, and $t^{(r)}_{e,n}, t^{(r)}_{s,n}$ are predictable stopping times.

In our formulation of a Markov and counting processes in \sectn{stationary}, this model has $d=1$ and $k=2$. We assume the following assumptions, in which we write $h(r) = o(r)$ for a function $h$ of variable $r$ if $\lim_{r \downarrow 0} h(r)/r = 0$.
\begin{mylist}{3}
\item [(4.a)] $\{(T^{(r)}_{e})^{2}; r \in (0,1]\}$ and $\{(T^{(r)}_{s})^{2}; r \in (0,1]\}$ are uniformly integrable.
\item [(4.b)] $m^{(r)}_{e}, \lambda^{(r)}, \sigma^{(r)}_{e}$ converge to $m_{e}, \lambda, \sigma_{e}$, respectively, as $r \downarrow 0$. Similarly, $m^{(r)}_{s}, \mu^{(r)}, \sigma^{(r)}_{s}$ converge to $m_{s}, \mu, \sigma_{s}$, respectively, as $r \downarrow 0$.
\item [(4.c)] There is a $b \in \dd{R}$ such that $\mu^{(r)} - \lambda^{(r)} = r \mu b + o(r)$ as $r \downarrow 0$. 
\item [(4.d)] $r \ell^{(r)}_{0} = \ell_{0}+ o(r)$.
\end{mylist}
Note that (4.b) and (4.c) imply $\lambda = \mu$ and $\rho \equiv \lambda/\mu = 1$. Furthermore, $1 - \rho^{(r)} = rb + o(r)$.

Assume that $X^{(r)}(\cdot)$ has the stationary distribution, and let $X^{(r)} \equiv (L^{(r)},R^{(r)}_{1},R^{(r)}_{2})$ be a random vector subject to this stationary distribution. Let $\alpha^{(r)}_{i} = \dd{E}[N^{(r)}_{i}(1)]$ for $i=1,2$ under the stationary framework. Obviously, $\alpha^{(r)}_{1} = \lambda^{(r)}$, but $\alpha^{(r)}_{2}$ is not $\mu^{(r)}$ but must be the arrival rate of customers who can get service. Hence, we only know that $\alpha^{(r)}_{2} \le \lambda^{(r)} < \infty$ at this moment, but we can define the Palm distribution $\dd{P}^{(r)}_{i}$ concerning point process $N^{(r)}_{i}$ for $i=1,2$. Then, we can see that
\begin{align}
\label{eq:alpha-2 1}
  \alpha^{(r)}_{2} & = \lim_{t \to \infty} \frac{N^{(r)}_{1}(t)}{t} \frac 1{N^{(r)}_{1}(t)} \int_{0}^{t} 1(L^{(r)}(u-) < \ell^{(r)}_{0}) N^{(r)}_{1}(du)  \nonumber\\
  & = \lambda^{(r)} \dd{P}^{(r)}_{1}(L^{(r)}(0-) < \ell^{(r)}_{0}),
\end{align}
which is not a concrete expression for $\alpha^{(r)}_{2}$, but will be used to prove \thr{FQ-limit 1} below.

We aim to derive the limiting distribution of $rL^{(r)}(t)$ as $r \downarrow 0$. We have the following answer to this problem.

\begin{theorem}\rm
\label{thr:FQ-limit 1}
For the $GI/G/1/\ell_{0}$ queue, assume that $X^{(r)}(\cdot)$ has the stationary distribution and the conditions (4.a), (4.b), (4.c) and (4.d) are satisfied, then the limiting distribution of $rL^{(r)}$ exists as $r \downarrow 0$, and has a density. Denote this density function by $g$, then we have two cases.
\begin{enumerate}[(i)]
\item For $b=0$, 
\begin{align}
\label{eq:null 1}
  \dd{P}(rL^{(r)} = 0) = \dd{P}^{(r)}_{1} (r L^{(r)}(0-) = \ell^{(r)}_{0}) = \frac 1{2\ell_{0}} (\lambda^{2} \sigma_{e}^{2} + \mu^{2} \sigma_{s}^{2})r + o(r), 
\end{align}
and $g$ is uniform on $[0,\ell_{0}]$.
\item For $b \not= 0$, let $\ds \beta = \frac {2b}{\lambda^{2}\sigma_{e}^{2} + \mu^{2}\sigma_{e}^{2}}$, then 
\begin{align}
\label{eq:FQ-lower 1}
 & \dd{P}(rL^{(r)} = 0) = \frac {be^{\beta\ell_{0}}}{e^{\beta\ell_{0}} - 1} r + o(r),\\
\label{eq:FQ-upper 1}
 & \dd{P}^{(r)}_{1} (L^{(r)}(0-) = \ell^{(r)}_{0}) = \frac {b}{e^{\beta\ell_{0}} - 1} r + o(r),
\end{align}
and $g$ is a truncated exponential function on $[0,\ell_{0}]$ with parameter $\beta$. Namely,
\begin{align}
\label{eq:FQ-stationary 2}
  g(x) = \frac {\beta}{1 - e^{-\beta \ell_{0}}} e^{-\beta x}, \qquad x \in [0,\ell_{0}],
\end{align}
where $g(x)$ is decreasing for $b > 0$, while it is increasing for $b < 0$.
\end{enumerate}
\end{theorem}

\begin{remark}\rm
\label{rem:FQ-limit 1}
Since $L^{(r)}(t)$ is bounded by $\ell^{(r)}_{0}$, $X^{(r)}(\cdot)$ has a unique stationary distribution under a mild regularity condition such that the distribution of $T_{e}$ is spread out (e.g., see \cite{Asmu2003} for this condition and \cite{MiyaMoro2023} for the stability).
\pend
\end{remark}

\begin{remark}\rm
\label{rem:FQ-limit 2}
The distributions obtained in (i) and (ii) agree with those of the corresponding reflecting Brownian process \cite[Proposition 6.6]{Harr2013}. If this Brownian process is obtained as the weak limit of the sequence of $rX^{(r)}(r^{-2}t)$ as $r \downarrow 0$, so called a process limit in diffusion scaling, then (i) and (ii) are immediate from the tightness of the sequence of the stationary distributions because their supports are uniformly bounded. Thus, (i) and (ii) may not be new results. However, the present approach skips the derivation of the process limit. Furthermore, it derives the finer asymptotic results, \eq{null 1}, \eq{FQ-lower 1} and \eq{FQ-upper 1}, which may not be obtained through the process limit. These are the advantages of the BAR approach.
\pend
\end{remark}

\begin{proof}
We first derive a BAR for this queueing model. Our first job is to find a good test function for the BAR \eq{stationary 2}. Similar to the $GJ$-network case (see \eq{GJN-test-f 1}), we take the following test function $f_{\theta,r}$ parametrized by $\theta \in \dd{R}$ and $r \in (0,1]$.
\begin{align*}
  f_{\theta,r}(X^{(r)}) = e^{\theta L^{(r)} - \eta^{(r)}(\theta) (R^{(r)}_{1} \wedge 1/r) - \xi^{(r)}(\theta) (R^{(r)}_{2} \wedge 1/r)},
\end{align*}
where $\eta^{(r)}(\theta), \xi^{(r)}(\theta)$ are uniquely determined by
\begin{align*}
  e^{\theta} \dd{E}(e^{- \eta^{(r)}(\theta) (T^{(r)}_{e} \wedge 1/r)}) = 1, \qquad e^{-\theta} \dd{E}(e^{- \xi^{(r)}(\theta) (T^{(r)}_{s} \wedge 1/r)}) = 1.
\end{align*}
We will use the test function $f_{r\theta,r}$, replacing $\theta$ by $r\theta$. As shown in \cite{BravDaiMiya2017} (see also \cite{BravDaiMiya2023}), $\eta^{(r)}(r\theta)$ and $\xi^{(r)}(r\theta)$ have the following asymptotic expansions as $r \downarrow 0$.
\begin{align}
\label{eq:}
 & \eta^{(r)}(r\theta) = \lambda^{(r)} r\theta + \frac 12 \lambda^{3} \sigma_{e}^{2} r^{2}\theta^{2} + o(r^{2}),\\
 & \xi^{(r)}(r\theta) = - \mu^{(r)} r\theta + \frac 12 \mu^{3} \sigma_{s}^{2} r^{2} \theta^{2} + o(r^{2}).
\end{align}
Furthermore, by Lemma 5.8 of \cite{BravDaiMiya2023}, there are constant $d_{e}, d_{s}, a > 0$ such that, for $r \in (0,1]$ and $\theta_{i} \in \dd{R}$ satisfying $r|\theta_{i}| < a$,
\begin{align}
\label{eq:FQ-eta-zeta 1}
  \left|\eta^{(r)}(r\theta) (u_{1} \wedge 1/r) + \zeta^{(r)}(r\theta) (u_{2} \wedge 1/r)\right| \le |\theta| \left(d_{e} (r u_{1} \wedge 1) + d_{s} (ru_{2} \wedge 1)\right).
\end{align}

Then, we can see from this and $r L^{(r)} \le r \ell^{(r)}_{0} \le \ell_{0} + o(r)$ that there is $r_{0} \in (0,1]$ for each $\theta \in \dd{R}$ such that
\begin{align}
\label{eq:FQ-finite 0}
 & \sup_{r (0,r_{0}]} f_{r\theta,r}(0,R^{(r)}) \le \max_{i=1,2} e^{|\theta_{i}|(d_{e,i} + d_{s})},\\
\label{eq:FQ-finite 1}
 & \sup_{r \in (0,r_{0}]} f_{r\theta,r}(L^{(r)},0,0) < \infty, \qquad \sup_{r \in (0,r_{0}]} f_{r\theta,r}(L^{(r)},R^{(r)}) < \infty.
\end{align}
Hence, $\dd{E} \left[f_{r\theta,r}(L^{(r)},R^{(r)})\right]$ and $\dd{E}^{(r)}_{i} \left[f_{r\theta,r}(L^{(r)},R^{(r)})\right]$ for $i=1,2$ are well defined, and
\begin{align}
\label{eq:FQ-limit 1}
  \lim_{r \downarrow 0} \dd{E} \left[f_{r\theta,r}(0,R^{(r)})\right] = 1, \qquad \lim_{r \downarrow 0} \dd{E}^{(r)}_{i}\left[f_{r\theta,r}(0,R^{(r)})\right] = 1, \qquad \theta \in \dd{R},
\end{align}
by the dominated convergence theorem, where recall that $\dd{P}^{(r)}_{i}$ is the Palm distribution for $i=1,2$, and $\dd{E}^{(r)}_{i}$ is expectations concerning $\dd{P}^{(r)}_{i}$. These facts legitimate computation below.

Let us compute the BAR \eq{stationary 2} for $f = f_{r\theta,r}$. Since $\dd{E}_{1}[\Delta f_{r\theta,r}(X^{(r)}(0)) 1(L^{(r)}(0-) \not= \ell^{(r)}_{0})] = 0$,
\begin{align*}
  \dd{E}^{(r)}_{1}[\Delta f_{r\theta,r}(X^{(r)}(0))] & = (e^{-r\theta} - 1) \dd{E}^{(r)}_{1}[f_{r\theta,r}(L^{(r)}(0-),R^{(r)}(0-)) 1(L^{(r)}(0-) = \ell^{(r)}_{0})] \nonumber\\
  & = (e^{-r\theta} - 1) e^{\theta \ell_{0}} \dd{E}^{(r)}_{1}[e^{-\xi^{(r)}(r\theta) (R^{(r)}_{2}(0-) \wedge 1/r)} 1(L^{(r)}(0-) = \ell^{(r)}_{0})]  \nonumber\\
  & = (e^{-r\theta} - 1) e^{\theta \ell_{0}} \dd{E}^{(r)}_{1}[1(L^{(r)}(0-) = \ell^{(r)}_{0}) f_{r\theta,r}(0,0,R^{(r)}_{2}(0-)) ],
\end{align*}
and $\dd{E}^{(r)}_{2}[\Delta f_{r\theta,r}(X^{(r)}(0))] = 0$. Furthermore, $(R^{(r)}_{1})'(t) = - 1$ and $(R^{(r)}_{2})'(t) = - 1(L^{(r)}(t) \ge 1)$. Hence, \eq{stationary 2} yields
\begin{align}
\label{eq:BAR 1}
 & \left(-\mu b r^{2} \theta + \frac 12 (\lambda^{3} \sigma_{e}^{2} + \mu^{3} \sigma_{s}^{2}) r^{2}\theta^{2} + o(r^{2}) \right) \dd{E}[e^{\theta r L^{(r)}}f_{r\theta,r}(0,R^{(r)})] \nonumber\\
 & \qquad + \left(\mu r\theta - \frac 12 \mu^{3} \sigma_{s}^{2} r^{2}\theta^{2} + o(r^{2})\right) \dd{E}[ 1(L^{(r)} = 0) f_{r\theta,r}(0,R^{(r)})] \nonumber\\
 & \qquad - (r\theta +O(r^{2})) e^{r\theta \ell^{(r)}_{0}} \lambda^{(r)} \dd{E}^{(r)}_{1}[1(L^{(r)}(0-) = \ell^{(r)}_{0}) f_{r\theta,r}(0,0,R^{(r)}_{2}(0-)) ] = 0.
\end{align}
This is an asymptotic BAR for our analysis. In this formula, $f_{r\theta,r}(0,R^{(r)})$ under $\dd{E}$ and $\dd{E}^{(r)}_{1}$ can be replaced by constant $1 + o(1)$ because of \eq{FQ-limit 1}. However, $\dd{P}(L^{(r)} = 0)$ and $\dd{P}^{(r)}_{1}(L^{(r)}(0-) = \ell^{(r)}_{0})$ are unknown, so we need to see their asymptotic behavior.

Recall the definition of $\beta$ in (ii) and note that $\theta = \beta$ is the solution of the equation $- \mu b + \frac 12 (\lambda^{3} \sigma_{e}^{2} + \mu^{3} \sigma_{s}^{2}) \theta = 0$. Since $\lambda/\mu = 1$, \eq{BAR 1} with $\theta = \beta$ yields
\begin{align}
\label{eq:beta 2}
  \dd{P}(L^{(r)}=0) = (e^{\beta \ell_{0}} + O(r)) \rho^{(r)} \dd{P}^{(r)}_{1} (L^{(r)}(0-) = \ell^{(r)}_{0})) + o(r),
\end{align}
where $h(r) = O(r)$ if $\limsup_{r \downarrow 0} |h(r)|/r$ is finite for a function $h$ with variable $r$. We next apply $f(X^{(r)}(t)) = R^{(r)}_{2}(t)$ to \eq{stationary 2}, then
\begin{align*}
 & - \dd{P}(L^{(r)} > 0) + \alpha^{(r)}_{2} \dd{E}^{(r)}_{2} [T^{(r)}_{s}] = 0.
\end{align*}
By \eq{alpha-2 1} and $\dd{E}^{(r)}_{2} [T^{(r)}_{s} ] = (\mu^{(r)})^{-1}$, this formula yields
\begin{align*}
  \lambda^{(r)} (1 - \dd{P}^{(r)}_{1} (L^{(r)}(0-) = \ell^{(r)}_{0})) = \alpha^{(r)}_{2} = \mu^{(r)} (1-\dd{P}(L^{(r)}=0)).
\end{align*}
Hence, 
\begin{align}
\label{eq:alpha 3}
  1 - \rho^{(r)} = \dd{P}(L^{(r)}=0) - \rho^{(r)} \dd{P}^{(r)}_{1} (L^{(r)}(0-) = \ell^{(r)}_{0}).
\end{align}

Define functions $\varphi^{(r)}$ and $\psi^{(r)}$ as
\begin{align*}
  \varphi^{(r)}(\theta) = \dd{E}[e^{\theta r L^{(r)}}], \qquad \psi^{(r)}(\theta) = \dd{E}[e^{\theta r L^{(r)}}f_{r\theta,r}(0,R^{(r)})], \qquad \theta \in \dd{R}, 
\end{align*}
then, by \eq{FQ-finite 1} and \eq{FQ-limit 1}, 
\begin{align}
\label{eq:FQ-finite 2}
  \lim_{r \downarrow 0} |\varphi^{(r)}(\theta) - \psi^{(r)}(\theta)| = 0, \qquad \theta \in \dd{R}.
\end{align}
Hence, if the limit of $\psi^{(r)}(\theta)$ is obtained as $r \downarrow 0$, then $\varphi^{(r)}(\theta)$ has the same limit. Denote this limit by $\widetilde{\varphi}(\theta)$ if it exists. If we can show that $\widetilde{\varphi}(\theta)$ is the moment generating function of a probability distribution on $\dd{R}_{+}$, then, denoting this distribution by $\widetilde{\nu}$, it can is shown that the distribution of $rL^{(r)}$ weakly converges to $\widetilde{\nu}$ as $r \downarrow 0$ by the weak convergence of a family of distributions under equicontinuity (e.g., see Lemma 5.2 of \cite{Kall2001}). 

To compute $\widetilde{\varphi}(\theta)$, we separately consider the asymptotic BAR \eq{BAR 1} for $b = 0$ and $b \not= 0$. We first assume that $b = 0$.  In this case, it follows from \eq{BAR 1} and \eq{alpha 3} that
\begin{align}
\label{eq:FQ-BAR 4}
  \frac {\lambda^{2} (\sigma_{e}^{2} + \sigma_{s}^{2})}{2} r \varphi^{(r)}(\theta) + \frac {1-e^{r\theta \ell^{(r)}_{0}}}{\theta} \dd{P}(L^{(r)}=0) = o(r).
\end{align}
Since this formula uniformly holds over $r \in (0,1]$ for $\theta$ such that $|r \theta| < a$ for each finite $a$, letting $\theta \to 0$, we have
\begin{align*}
   \frac {\lambda^{2} (\sigma_{e}^{2} + \sigma_{s}^{2})}{2} r = r \ell^{(r)}_{0} \dd{P}(L^{(r)}=0) + o(r).
\end{align*}
Hence, we have \eq{null 1}, and therefore, it follows from \eq{FQ-finite 2} and \eq{FQ-BAR 4} that
\begin{align}
\label{eq:FQ-limit 3}
  \lim_{r \downarrow 0} \varphi^{(r)}(\theta) = \lim_{r \downarrow 0} \psi^{(r)}(\theta) = \frac {e^{\theta \ell_{0}} - 1}{\ell_{0}},
\end{align}
which shows that $\widetilde{\nu}$ is the uniform distribution on $[0,\ell_{0}]$. Thus, (i) is proved.

We next assume that $b \not= 0$. Substituting $\dd{P}(L^{(r)}=0)$ of \eq{beta 2} into \eq{alpha 3}, it follows from $1 - \rho^{(r)} = rb + o(r)$ that
\begin{align*}
  br = (e^{\beta \ell_{0}} + O(r)) \rho^{(r)} \dd{P}^{(r)}_{1} (L^{(r)}(0-) = \ell^{(r)}_{0})) - \rho^{(r)} \dd{P}^{(r)}_{1} (L^{(r)}(0-) = \ell^{(r)}_{0}) + o(r)
\end{align*}
yields
\begin{align}
\label{eq:BAR 3}
  (e^{\beta \ell_{0}} - 1 + O(r)) \dd{P}^{(r)}_{1} (L^{(r)}(0-) = \ell^{(r)}_{0})) = br + o(r).
\end{align}
From this and substituting this into \eq{beta 2}, we have
\begin{align}
\label{eq:upper 1}
 & \dd{P}^{(r)}_{1} (L^{(r)}(0-) = \ell^{(r)}_{0}) = \frac {b}{e^{\beta\ell_{0}} - 1} r + o(r),\\
\label{eq:lower 1}
 & \dd{P}(L^{(r)} = 0) = \frac {be^{\beta\ell_{0}}}{e^{\beta\ell_{0}} - 1} r + o(r).
\end{align}

Then, substituting \eq{upper 1} and \eq{lower 1} into \eq{BAR 1} and letting $r \downarrow 0$ after dividing by $\theta r^{2}$, we have, for $\theta \not= \beta$,
\begin{align}
\label{eq:FQ-limit 2}
  \lim_{r \downarrow 0} \psi^{(r)}(\theta) = \lim_{r \downarrow 0} \frac {\dd{P}(L^{(r)}=0) - e^{\theta \ell_{0}} \dd{P}^{(r)}_{1}[L^{(r)}(0-)=\ell^{(r)}_{0})} {(\beta - \theta) \lambda^{2} (\sigma_{e}^{2} + \sigma_{s}^{2})/2} = \frac {(e^{\beta \ell_{0}} - e^{\theta \ell_{0}}) \beta} {(e^{\beta \ell_{0}} - 1)(\beta - \theta)}.
\end{align}
Hence, 
\begin{align*}
  \widetilde{\varphi}(\theta) = \frac {(e^{\beta \ell_{0}} - e^{\theta \ell_{0}}) \beta} {(e^{\beta \ell_{0}} - 1)(\beta - \theta)}, \qquad \theta \not= \beta.
\end{align*}
Obviously, this is the moment generating function of the distribution whose density is $g$ of \eq{FQ-stationary 2}. Thus, the proof of (ii) is completed.
\end{proof}

\subsection*{Acknowledgements}

The author thanks two anonymous referees for their helpful comments and suggestions. This paper came out from discussions with Jim Dai on \cite{BravDaiMiya2023}. I am grateful to Jim Dai for stimulating discussions. This paper is written based on my talk at the conference, 40 years of reflected Brownian motion and related topics, in France, April, 2023. The author thanks for the organizer of this conference.



\end{document}